\documentclass{smathjsty}

\usepackage{amsmath,amsthm,amscd,amsfonts,amssymb,amsxtra}
\usepackage[mathscr]{euscript}

\newtheorem{theorem}{Theorem}
\newtheorem{definition}{Defiition}

\newcommand{\sS}{\mathscr{S}}
\newcommand{\G}{\mathscr{G}}

\newcommand{\s}{\sigma}
\newcommand{\hd}{\hat{d}}

\begin{document}

\DateofSubmission{October 4, 2019}
\EmailofCorrespondingAuthor{sudev.nk@christuniversity.in}

\title[Certain Topological Indices of Signed Graphs] {On Certain Topological Indices of Signed Graphs}
\author[N. K. Sudev and J.Kok] {Sudev Naduvath$^{1,*}$, Johan Kok$^1$ }
\maketitle

\vspace*{-0.2cm}

\begin{center}
{\footnotesize  
$^1$Department of Mathematics, CHRIST (Deemed to be University), Bangalore 560029, India.\\
}
\end{center}

{\tiny\CopyrightInfo}

{\footnotesize 

\noindent {\bf Abstract.}  The first Zagreb index of a graph $G$ is the sum of squares of the vertex degrees in a graph and the second Zagreb index of $G$ is the sum of products of degrees of adjacent vertices in $G$. The imbalance of an edge in $G$ is the numerical difference of degrees of its end vertices and the irregularity of $G$ is the sum of imbalances of all its edges. In this paper, we extend the concepts of these topological indices for signed graphs and discuss the corresponding results on signed graphs. \vskip 1mm

\vskip 1mm


}

\vskip 6mm

\FootNoteInfo 

\section{Introduction}

For general notation and concepts in graph theory, we refer to \cite{BM1,FH1,DBW} and for the terminology in signed graphs, see \cite{TZ1,TZ2}. Unless mentioned otherwise, all graphs considered here are finite, simple, undirected and connected.

Let $G(V, E)$ be a graph with vertex set $V(G)=\{v_1, v_2,\ldots, v_n\}$ and
edge set $E(G)$, where $|E(G)|=m$. The degree of a vertex $v_i$ in $G$ is the number of edges incident on it and is denoted by $d_G(v_i)$. If the context is clear, let us use the notation $d_i$ instead of $d_G(v_i)$.

The \textit{first Zagreb index}, denoted by $M_1(G)$ and the \textit{second Zagreb index}, denoted by $M_2(G)$ is defined in \cite{GT1} as $M_1(G) = \sum\limits_{v_i\in V(G)} d_i^2$ and $M_2(G) = \sum\limits_{v_iv_j\in E(G)}^n d_id_j$. The \textit{imbalance} of an edge $e=uv \in E(G)$ is defined as $imb_G(uv)=|d_G(u)-d_G(v)|$ (see \cite{MOA1}). The notion of \textit{irregularity} of a graph $G$ has also been introduced in \cite{MOA1} as $irr(G)=\sum\limits_{uv\in E(G)}imb(uv)$. Another new measure of irregularity of a simple undirected graph $G$, called the \textit{total irregularity} of $G$, denoted by $irr_t(G)$, is defined in \cite{ABD1} as $irr_t(G)=\frac{1}{2}\sum\limits_{u,v\in V(G)}|d(u)-d(v)|$.

We extend the notions of these topological indices of graphs defined in the previous section to the theory of signed graphs. 


A \textit{signed graph} (see \cite{TZ1,TZ2}), denoted by $S(G,\s)$,  is a graph $G(V,E)$ together with a function $\s:E(G)\to \{+,-\}$ that assigns a sign, either $+$ or $-$, to each ordinary edge in $G$. The function $\s$ is called the {\em signature} or {\em sign function} of $S$, which is defined on all edges except half edges and is required to be positive on free loops. The unsigned graph $G$ is called the \textit{underlying graph} of the signed graph $S$.

An edge $e$ of $S$ is said to be \textit{positive} or \textit{negative} in accordance with its signature $\s(e)$ is positive or negative. The number of positive edges incident on a vertex $v$ in $S$ is the \textit{positive degree} of $v$ and is denoted by $d_S^+(v)$ and the number of negative edges incident on $v$ is the \textit{negative degree} of $v$ and is denoted by $d_S^-(v)$. Clearly, $d_G(v)=d_S^+(v)+d_S^-(v)$.

Analogous to the definition of first Zagreb index of a graph, we can define two types of first Zagreb indices for a given signed  graph $S$ as follows.

\begin{definition}{\rm 
Let $S$ be a signed graph and let $d_i^+$ and $d_i^-$ be the positive and negative degree of a vertex $v_i$ in $S$. Then, the \textit{first positive Zagreb index} of $S$ is denoted by $M_1^+(S)$ is defined as
\begin{gather}\label{Eqn-2.2}
M_1^+(S)=\sum\limits_{v_i\in V(G)}\left(d_i^+\right)^2,
\end{gather}
the \textit{first negative Zagreb index} of $S$ is denoted by $M_1^-(S)$ is
\begin{gather}\label{Eqn-2.3}
M_1^-(S)=\sum\limits_{v_i\in V(G)}\left(d_i^-\right)^2
\end{gather}
and the \textit{first mixed Zagreb index} of a signed graph $S$ is denoted by $M_1^{\ast}$ and is defined as
\begin{gather}\label{Eqn-2.4}
M_1^{\ast}(S)=\sum\limits_{v_i\in V(G)} d_i^+d_i^-
\end{gather}
}\end{definition}

In a similar way, we can also define the second Zagreb indices for a signed graph $S$ as follows.

\begin{definition}{\rm 
Let $S$ be a signed graph and let $d_i^+$ and $d_i^-$ be the positive and negative degree of a vertex $v_i$ in $S$. Then, the \textit{second positive Zagreb index} of $S$ is denoted by $M_2^+(S)$ is defined as
\begin{gather}\label{Eqn-2.5}
M_2^+(S)=\sum\limits_{v_iv_j\in E(G)}d_i^+d_j^+;\ 1\le i\ne j \le n,
\end{gather}
the \textit{second negative Zagreb index} of $S$ is denoted by $M_2^-(S)$ is
\begin{gather}\label{Eqn-2.6}
M_2^-(S)=\sum\limits_{v_iv_j\in E(G)}d_i^-d_j^-;\ 1\le i\ne j \le n
\end{gather}
and the \textit{second mixed Zagreb index} of a signed graph $S$ is denoted by $M_2^{\ast}$ and is defined as
\begin{gather}\label{Eqn-2.7}
M_2^{\ast}(S)=\sum\limits_{v_iv_j\in E(G)} d_i^+d_j^-;\ 1\le i\ne j \le n.
\end{gather}
}\end{definition}

In view of the new notions defined above, the relation between the first Zagreb indices of a signed graph $S$ and the first Zagreb index of its underlying graph $G$ is discussed in the following theorem.

\begin{theorem}
For a signed graph $S$ and its underlying graph $G$,
\begin{enumerate}\itemsep0mm
\item[(i)] $M_1(G)=M_1^+(S)+ M_1^-(S)+2M_1^{\ast}(S)$,
\item[(ii)] $M_2(G)=M_2^+(S)+ M_2^-(S)+M_2^{\ast}(S)$.
\end{enumerate}
\end{theorem}
\begin{proof}
Let $d_i$ denotes the degree of a vertex $v_i$ in the underlying graph $G$ of a signed graph $S$. Then, we have

\begin{eqnarray*}
\noindent {\rm (i)}\ \  M_1(G) & = & \sum\limits_{v_i\in V(G)} (d_i)^2\\
& = & \sum\limits_{v_i\in V(G)} (d_i^+ + d_i^-)^2\\
& = & \sum\limits_{v_i\in V(G)} (d_i^+)^2 + \sum\limits_{v_i\in V(G)} (d_i^-)^2 + 2\sum\limits_{v_i\in V(G)} d_i^+d_i^-\\
& = & M_1^+(S)+ M_1^-(S)+2\sum\limits_{v_i\in V(G)} d_i^+d_i^-\\
M_1(G) & = & M_1^+(S)+ M_1^-(S)+2 M_1^{\ast}(S). \\\\
\noindent {\rm(ii)}\ \  M_2(G) & = & \sum\limits_{v_iv_j\in E(G)} d_id_j\\ 
& = & \sum\limits_{v_iv_j\in E(G)} (d_i^+ + d_i^-)(d_j^+ + d_j^-)\\
& = & \sum\limits_{v_iv_j\in E(G)} d_i^+d_j^+ + \sum\limits_{v_iv_j\in E(G)} d_i^-d_j^- + \sum\limits_{v_iv_j\in E(G)} d_i^+d_j^-\\
& = & M_2^+(S)+ M_2^-(S)+M_2^{\ast}(S).
\end{eqnarray*}
\vspace{-1cm}

\end{proof}


Analogous to the definition of imbalance of edges in graphs, let us introduce the following definitions for signed graphs.

\begin{definition}{\rm 
For an edge $e=uv$ in a signed graph $S$, the \textit{positive imbalance} of $e$ can be defined as $imb_S^+(uv)=|d_S^+(u)-d_S^+(v)|$ and the \textit{negative imbalance} of $e$ can be defined as $imb_S^-(uv)=|d_S^-(u)-d_S^-(v)|$.}
\end{definition} 

In a similar way, the two types irregularities of a signed graph can be defined as follows.

\begin{definition}{\rm 
The \textit{positive irregularity} of a signed graph $S$, denoted by $irr^+(S)$, is defined to be $irr^+(S)=\sum\limits_{uv\in E(S)}imb_S^+(uv)$ and the \textit{negative irregularity} of $S$, denoted by $irr^-(S)$, is defined as $irr^-(S)=\sum\limits_{uv\in E(S)}imb_S^-(uv)$.}
\end{definition} 

The \textit{total irregularities} of a signed graph $S$ can also be defined as given below.

\begin{definition}{\rm 
The \textit{total positive irregularity} of a signed graph $S$, denoted by $irr_t^+(S)$, is defined as 
$$irr_t^+(G)=\frac{1}{2}\sum\limits_{u,v\in V(S)}|d_S^+(u)-d_S^+(v)|.$$
and the \textit{total negative irregularity} of a signed graph $S$, denoted by $irr_t^+(S)$, is defined as 
$$irr_t^-(G)=\frac{1}{2}\sum\limits_{u,v\in V(S)}|d_S^-(u)-d_S^-(v)|.$$ 
}\end{definition}

The following theorem discusses the relation between the irregularities and total irregularities of a signed graph $S$ with the corresponding indices of its underlying graph $G$.

\begin{theorem}\label{Thm-2.2}
Let $S$ be a signed graph and $G$ denotes its underlying graph. Then, we have
\begin{enumerate}
\item[(i)] $imb_G(uv) \le imb_S^+(uv)+imb_S^-(uv);\ uv\in E(S)\ (and\ E(G))$,
\item[(ii)] $irr(G) \le irr^+(S)+irr^-(S)$,
\item[(iii)] $irr_t(G) \le irr_t^+(S)+irr_t^-(S)$.
\end{enumerate}
\end{theorem}
\begin{proof}
Let $S$ be a signed graph with underlying graph $G$ and let $e=uv$ be any edge of $S$ (and $G$). Then, 
\begin{eqnarray*}
imb_G(uv) & = & |d_G(u)-d_G(v)|\\
& = & |\left(d^+(u)+d^-(u)\right)-\left(d^+(v)+d^-(v)\right)|\\
& = & |\left(d^+(u)-d^+(v)\right)+\left(d^-(u)-d^-(v)\right)|\\
& \le & |\left(d^+(u)-d^+(v)\right)|+|\left(d^-(u)-d^-(v)\right)|\\
\therefore\ imb_G(uv) & \le & imb_S^+(uv)+imb_S^-(uv).
\end{eqnarray*}
Also, 
\begin{eqnarray*}
irr(G) & = & \sum\limits_{uv\in E(G)} imb_G(uv)\\
& \le & \sum\limits_{uv\in E(S)} \left(imb_S^+(uv)+imb_S^-(uv)\right)\\
& = & \sum\limits_{uv\in E(S)} imb_S^+(uv)+ \sum\limits_{uv\in E(S)} imb_S^-(uv)\\
\therefore\ irr(G) & \le & irr^+(S)+irr^-(S).
\end{eqnarray*}
and 
\begin{eqnarray*}
irr_t(G) & = & \frac{1}{2}\sum\limits_{u,v\in V(G)}|d_G(u)-d_G(v)|\\
& = & \frac{1}{2}\left(\sum\limits_{u,v\in V(G)} \left|\left(d_S^+(u)+d_S^-(u)\right)-\left(d_S^+(v)+d_S^-(v)\right)\right|\right)\\
& = & \frac{1}{2}\left(\sum\limits_{u,v\in V(G)}\left|\left(d_S^+(u)-d_S^+(v)\right)+\left(d_S^-(u)+d_S^-(v)\right)\right|\right)\\
& \le & \frac{1}{2}\sum\limits_{u,v\in V(G)}\left|\left(d_S^+(u)-d_S^+(v)\right)\right|+\frac{1}{2}\sum\limits_{u,v\in V(G)}\left|\left(d_S^-(u)-d_S^-(v)\right)\right|\\
\therefore\ irr_t(G) & \le & irr_t^+(S)+irr_t^-(S).
\end{eqnarray*}
\vspace{-1cm}

\end{proof}


The \textit{net-degree} of a signed graph $S$, denoted by $d_S^{\pm}$, is defined in \cite{HPG1} as $d_S^{\pm}(v)=d_S^+(v)-d_S^-(v)$. The signed graph $S$ is said to be \textit{net-regular} if every vertex of $S$ has the same net-degree. Different from the notation used in \cite{HPG1}, we use notation $\hat{d}_S(v)$ represent the net-degree of a vertex in a signed graph $S$. 

Invoking the notion of the net-degree of vertices in a signed graph $S$, we introduce the following notions on $S$.

\begin{definition}{\rm 
Let $S$ be a signed graph and let $\hd_i$ denotes the net-degree of a vertex in $S$. The \textit{first Zagreb net-index} of the signed graph $S$ is denoted by $M_1(S)$ and is defined as 
\begin{gather}
M_1(S)= \sum\limits_{v_i\in V(S)} \hd_i^2
\end{gather}
and the \textit{second Zagreb net-index} of $S$ is denoted by $M_2(S)$ is defined as 
\begin{gather}
M_2(S)= \sum\limits_{v_iv_j\in E(S)} \hd_i\hd_j
\end{gather}
}\end{definition}

\vspace{0.25cm}

\noindent In view of the above notions, we have the following theorems.

\begin{theorem}\label{Thm-3.1}
Let $S$ be a signed graph and $G$ be its underlying graph. Then, 
\begin{enumerate}\itemsep0mm
\item[(i)] $M_1(G)=M_1(S)+4M_1^{\ast}(S)$,
\item[(ii)] $M_2(G)=M_2(S)+2M_2^{\ast}(S)$.
\end{enumerate}
\end{theorem}
\begin{proof}
Let $d_i,d_i^+, d_i^-$ respectively represent the degree, positive degree and negative degree of a vertex $v_i$ in $S$. Then, we have  $\hd_i=d_i^+-d_i^-$. Then, we have
\begin{eqnarray*}
M_1(S) & = & \sum\limits_{v_i\in V(S)}\hd_i^2\\
& = & \sum\limits_{v_i\in V(S)}(d_i^+-d_i^-)^2\\
& = & \sum\limits_{v_i\in V(S)}(d_i^+)^2+\sum\limits_{v_i\in V(S)}(d_i^-)^2-2\sum\limits_{v_i\in V(S)}d_i^+d_i^-\\
& = & M_1^+(S)+M_1^-(S)-2M_1^{\ast}(S)\\
& = & (M_1(G)-2M_1^{\ast}(S))-2M_1^{\ast}(S)\\
& = & M_1(G)-4M_1^{\ast}(S).
\end{eqnarray*}
Similarly, 
\begin{eqnarray*}
M_2(S) & = & \sum\limits_{v_iv_j\in E(S)}\hd_i\hd_j\\
& = & \sum\limits_{v_iv_j\in E(S)}(d_i^+-d_i^-)(d_j^+-d_j^-)\\
& = & \sum\limits_{v_iv_j\in E(S)}d_i^+d_j^+ +\sum\limits_{v_iv_j\in E(S)} d_i^-d_j^- -\sum\limits_{v_iv_j\in E(S)}d_i^+d_j^-;\ i\ne j\\
& = & M_2^+(S)+M_2^-(S)-M_2^{\ast}(S)\\
& = & (M_2(G)-M_2^{\ast}(S))-M_2^{\ast}(S)\\
& = & M_2(G)-2M_2^{\ast}(S).
\end{eqnarray*}
\end{proof}

Analogous to the definition of imbalance and irregularities of signed graphs mentioned in the previous section, we introduce the following notions.

\begin{definition}{\rm 
For an edge $e=uv$ in a signed graph $S$, the \textit{net-imbalance} or simply the \textit{imbalance} of $e$, denoted by $imb_S(uv)$, is defined as $imb_S(uv)=|\hd_S(u)-\hd_S(v)|$.}
\end{definition}

\begin{definition}{\rm 
The \textit{irregularity} of a signed graph $S$, denoted by $irr(S)$, is defined to be 
\begin{gather}
irr(S)=\sum\limits_{uv\in E(S)}imb_S(uv)=\sum\limits_{uv\in E(S)}|\hd_S(u)-\hd_S(v)|.
\end{gather}
}\end{definition} 

\begin{definition}{\rm 
The \textit{total irregularity} of a signed graph $S$, denoted by $irr_t(S)$, is defined as 
\begin{gather}
irr_t(G)=\frac{1}{2}\sum\limits_{u,v\in V(S)}|\hd_S(u)-\hd_S(v)|
\end{gather}
}\end{definition}

\noindent Note that if a signed graph $S$is net-regular, then $\hd_i=\hd_j;\ \forall\,i\ne j$ and hence we have $irr(S)=0$ and $irr_t(S)=0$.

\begin{theorem}\label{Thm-3.2}
For any signed graph $S$, we have
\begin{enumerate}\itemsep0mm
\item[(i)] $imb_S(v_iv_j)\ge imb_S^+(v_iv_j)-imb_S^-(v_iv_j)$,
\item[(ii)] $irr(S)\ge irr^+(S)-irr^-(S)$,
\item[(iii)] $irr_t(S)\ge irr_t^+(S)-irr_t^-(S)$.
\end{enumerate}
\end{theorem}
\begin{proof}
Let $e=v_iv_j$ be an arbitrary edge in $G$. Then
\begin{eqnarray*}
imb_S(v_iv_j) & = & |\hd_i-\hd_j|\\
& = & |(d_i^+-d_i^-)-(d_j^+-d_j^-)\\
& \ge & |d_i^+-d_j^+|-|d_i^--d_j^-|\\
& = & imb_S^+(v_iv_j)-imb_S^-(v_iv_j)\\
i.e.,\ imb_S(v_iv_j) & \ge & imb_S^+(S)-imb_S^-(S).
\end{eqnarray*}
Also, 
\begin{eqnarray*}
irr(S) & = & \sum\limits_{v_iv_j\in E(S)}|\hd_i-\hd_j|\\
& = & |\sum\limits_{v_iv_j\in E(S)}(d_i^+-d_i^-)-(d_j^+-d_j^-)\\
& \ge & \sum\limits_{v_iv_j\in E(S)}|d_i^+-d_j^+|-\sum\limits_{v_iv_j\in E(S)}|d_i^--d_j^-|\\
& = & \sum\limits_{v_iv_j\in E(S)} imb_S^+(v_iv_j)-\sum\limits_{v_iv_j\in E(S)}imb_S^-(v_iv_j)\\
i.e.,\ imb_S(v_iv_j) & \ge & imb_S^+(S)-imb_S^-(S).
\end{eqnarray*}
and
\begin{eqnarray*}
irr_t(S) & = & \frac{1}{2}\sum\limits_{v_i,v_j\in V(S)}|\hd_i-\hd_j|\\
& = & \frac{1}{2}\sum\limits_{v_i,v_j\in V(S)}|(d_i^+-d_i^-)-(d_j^+-d_j^-)|\\
& = & \frac{1}{2}\sum\limits_{v_i,v_j\in V(S)}|(d_i^+-d_j^+)-(d_i^- -d_j^-)|\\
& \ge & \frac{1}{2}\sum\limits_{v_i,v_j\in V(S)}|d_i^+-d_j^+|-\frac{1}{2}\sum\limits_{v_i,v_j\in V(S)}|d_i^--d_j^-|\\
& = & irr_t^+(S)-irr_t^-(S)\\
i.e.,\ irr_t(S) & \ge & irr_t^+(S)-irr_t^-(S).
\end{eqnarray*}
\end{proof}


The \textit{Schultz index} of a graph $G$ is defined as $\sS(G)=\sum\limits_{u,v\in V}\left(d_G(u)+d_G(v)\right)d_G(u,v)$ (see \cite{HPS}). Analogous to this terminology, we introduce the following notions for signed graphs.

\begin{definition}{\rm 
The \textit{positive Schultz index} of a signed graph $S$, denoted by $\sS^+(S)$, is defined to be $$\sS^+(S)=\sum\limits_{u,v\in V(S)}\left(d^+_S(u)+d^+_S(v)\right)d_S(u,v)$$ and the \textit{negative Schultz index} of the signed graph $S$, denoted by $\sS^-(S)$, is defined to be $$\sS^-(S)=\sum\limits_{u,v\in V(S)}\left(d^-_S(u)+d^-_S(v)\right)d_S(u,v)$$.}
\end{definition}

Here, note that the distance between two vertices in a signed graph $S$ is the same as the distance between those two vertices in the underlying graph $G$ of $S$. In view of the above notions, we have the following theorem.

\begin{theorem}\label{Thm-4.1}
If $S$ be a signed graph and $G$ be its underlying graph, then $\sS(G)=\sS^+(S)+\sS^-(S)$.
\end{theorem}
\begin{proof}
Let $G$ be the underlying graph of a signed graph $S$. Then, for any two vertices $u,v\in V(S)$, $d_S(u,v)=d_G(u,v)$. Then, we have
\begin{eqnarray*}
\sS(G) & = & \sum\limits_{u,v\in V(S)}\left[d_G(u)+d_G(v)\right]d_G(u,v)\\
& = & \sum\limits_{u,v\in V(S)}\left[\left(d^+_S(u)+d^-_S(u)\right)+\left(d^+_S(v)+d^-_S(v)\right)\right]d_S(u,v)\\
& = & \sum\limits_{u,v\in V(S)}\left[\left(d^+_S(u)+d^+_S(v)\right)+\left(d^-_S(u)+d^-_S(v)\right)\right]d_S(u,v)\\
& = & \sum\limits_{u,v\in V(S)}\left[d^+_S(u)+d^+_S(v)\right]d_S(u,v)+ \sum\limits_{u,v\in V(S)}\left[d^-_S(u)+d^-_S(v)\right]d_S(u,v)\\
& = & \sS^+(S)+\sS^-(S).
\end{eqnarray*}
\end{proof}

Using the concepts of net-degree of vertices in a signed graph, we introduce the following notion.

\begin{definition}{\rm 
The \textit{Schultz index} of a signed graph $S$, denoted by $\sS(S)$, is defined as $$\sS(S)=\sum\limits_{u,v\in V(S)}\left(\hd_S(u)+\hd_S(v)\right)d_S(u,v)$$, where $\hd(v)$ is the net-degree of a vertex $v\in V(S)$.
}\end{definition}

Invoking the above definition, we have the following theorem on the Schultz index of signed graphs.

\begin{theorem}\label{Thm-4.2}
For a signed graph $S$, then $\sS(S)=\sS^+(S)-\sS^-(S)$.
\end{theorem}
\begin{proof}
Let $G$ be the underlying graph of a signed graph $S$. Then, as mentioned in the previous theorem, for any two vertices $u,v\in V(S)$, $d_S(u,v)=d_G(u,v)$. Then, we have
\begin{eqnarray*}
\sS(S) & = & \sum\limits_{u,v\in V(S)}\left[\hd_S(u)+\hd_S(v)\right]d_S(u,v)\\
& = & \sum\limits_{u,v\in V(S)}\left[\left(d^+_S(u)-d^-_S(u)\right)+\left(d^+_S(v)-d^-_S(v)\right)\right]d_S(u,v)\\
& = & \sum\limits_{u,v\in V(S)}\left[\left(d^+_S(u)+d^+_S(v)\right)-\left(d^-_S(u)+d^-_S(v)\right)\right]d_S(u,v)\\
& = & \sum\limits_{u,v\in V(S)}\left[d^+_S(u)+d^+_S(v)\right]d_S(u,v)- \sum\limits_{u,v\in V(S)}\left[d^-_S(u)+d^-_S(v)\right]d_S(u,v)\\
& = & \sS^+(S)-\sS^-(S).
\end{eqnarray*}
\end{proof}

If $G$ is the underlying graph of a signed graph $S$, then we have $\sS(G)\ge \sS(S)$. Moreover, we have $\sS(G)+\sS(S)=2\sS^+(S)$ and $\sS(G)-\sS(S)=2\sS^-(S)$. 


The \textit{Gutman index} of a graph $G$ is another interesting topological index, denoted by $\G(G)$, which is defined as $\G(G)=\sum\limits_{u,v\in V}d_G(u)d_G(v)d_G(u,v)$ (see \cite{KG1}). Analogous to this terminology, we introduce the following notions for signed graphs.

\begin{definition}{\rm 
The \textit{positive Gutman index} of a signed graph $S$, denoted by $\G^+(S)$, is defined to be $$\G^+(S)=\sum\limits_{u,v\in V(S)}d^+_S(u)d^+_S(v)d_S(u,v).$$
The \textit{negative Gutman index} of the signed graph $S$, denoted by $\G^-(S)$, is defined as $$\G^-(S)=\sum\limits_{u,v\in V(S)}d^-_S(u)d^-_S(v)d_S(u,v).$$
The \textit{mixed Gutman index} of the signed graph $S$, denoted by $\G^*(S)$, is defined as $$\G^-(S)=\sum\limits_{u,v\in V(S)}d^+_S(u)d^-_S(v)d_S(u,v).$$
}\end{definition}

The following theorem discusses the relation between these Gutman indices of signed graphs and the Gutman index of its underlying graph.

\begin{theorem}\label{Thm-5.1}
If $G$ is the underlying graph of a signed graph $S$, then $\G(G)=\G^+(S)+\G^-(S)+\G^*(S)$.
\end{theorem}
\begin{proof}
\begin{eqnarray*}
\G(G) & = & \sum\limits_{u,v\in V(S)}d_G(u)d_G(v)d_G(u,v)\\
& = & \sum\limits_{u,v\in V(S)}\left[\left(d^+_S(u)+d^-_S(u)\right)\left(d^+_S(v)+ d^-_S(v)\right)\right]d_S(u,v)\\
& = & \sum\limits_{u,v\in V(S)}\left[d^+_S(u)d^+_S(v)+ d^+_S(u)d^-_S(v)+ d^-_S(u)d^+_S(v)+ d^-_S(u)d^-_S(v)\right]d_S(u,v)\\
& = & \sum\limits_{u,v\in V(S)}d^+_S(u)d^+_S(v)d_S(u,v)+ \sum\limits_{u,v\in V(S), u\ne v}d^+_S(u)d^-_S(v)d_S(u,v)+\\ &  &  \sum\limits_{u,v\in V(S)}d^-_S(u)d^-_S(v)d_S(u,v)\\
& = & \G^+(S)+\G^*(S)+\G^-(S).
\end{eqnarray*}
\end{proof}

Using the concepts of net-degree of vertices in a signed graph, we now introduce the following notion.

\begin{definition}{\rm 
The \textit{Gutman index} of a signed graph $S$, denoted by $\G(S)$, is defined as $$\G(S)=\sum\limits_{u,v\in V(S)}\hd_S(u)\hd_S(v)d_S(u,v),$$ where $\hd(v)$ is the net-degree of a vertex $v\in V(S)$.
}\end{definition}

Invoking the above definition, we have the following theorem on the Gutman index of signed graphs.

\begin{theorem}\label{Thm-5.2}
For a signed graph $S$, then $\G(S)=\G^+(S)+\G^-(S)-\sS^*(S)$.
\end{theorem}
\begin{proof}
\begin{eqnarray*}
\G(S) & = & \sum\limits_{u,v\in V(S)}\hd_S(u)\hd_S(v)d_S(u,v)\\
& = & \sum\limits_{u,v\in V(S)}\left[\left(d^+_S(u)-d^-_S(u)\right)\left(d^+_S(v)-d^-_S(v)\right)\right]d_S(u,v)\\
& = & \sum\limits_{u,v\in V(S)}\left[d^+_S(u)d^+_S(v)-d^+_S(u)d^-_S(v)-d^-_S(u)d^+_S(v)+d^-_S(u)d^-_S(v)\right]d_S(u,v)\\
& = & \sum\limits_{u,v\in V(S)}d^+_S(u)d^+_S(v)d_S(u,v)- \sum\limits_{u,v\in V(S),u\ne v}d^+_S(u)d^-_S(v)d_S(u,v)+\\ & & \sum\limits_{u,v\in V(S)}d^-_S(u)d^-_S(v)d_S(u,v)\\
& = & \G^+(S)-\G^*(S)+\G^-(S).
\end{eqnarray*}
\end{proof}

Several problems in this area are still open.  Determining various other topological indices for are yet to be determined. Investigations on various theoretical and practical applications of these indices of graphs and signed graphs are also promising. 

\ConflictofInterests

\end{document}